\documentclass[12pt]{article}
\usepackage{amsmath,amssymb, amsfonts, amsthm,amscd}
\usepackage[T2A]{fontenc}
\usepackage[cp1251]{inputenc}
\usepackage[english]{babel}
\usepackage{graphicx}

\theoremstyle{plain}
\newtheorem{theorem}{Theorem}
\newtheorem{lemma}{Lemma}

\newtheorem{definition}{Definition}
\theoremstyle{definition}

\begin{document}

\begin{center}
{\huge $J$-Noetherian  Bezout domain which is not a ring of stable range 1}
\end{center}
\vskip 0.1cm \centerline{{\Large Bohdan Zabavsky, \; Oleh Romaniv}}

\vskip 0.3cm

\centerline{\footnotesize{Department of Mechanics and Mathematics, Ivan Franko National University}}
\centerline{\footnotesize{Lviv, 79000, Ukraine}}
 \centerline{\footnotesize{zabavskii@gmail.com, \; oleh.romaniv@lnu.edu.ua}}
\vskip 0.5cm

\centerline{\footnotesize{August, 2018}}
\vskip 0.7cm

\footnotesize{\noindent\textbf{Abstract:} \textit{We proved that in $J$-Noetherian Bezout domain which is not a ring of stable range 1 exists a nonunit adequate element (element of almost stable range 1).} }

\vspace{1ex}
\footnotesize{\noindent\textbf{Key words and phrases:} \textit{$J$-Noetherian ring; Bezout ring; adequate ring;  elementary divisor ring;  stable range; almost stable range; neat range; maximal ideal.}

}

\vspace{1ex}

\vspace{1,5truecm}

\normalsize


Unique factorization domain are, of course, the integral domains in which every nonzero nonunit element has a unique factorization (up to order and associates) into irreducible elements, or atoms.  Now UFDs can also be characterized by the property that every nonzero nonunit is a product of principal primes or equivalently that every nonzero nonunit has the form $up_1^{\alpha_1} \cdot \ldots \cdot  p_n^{\alpha_n}$, where $u$ is a unit, $p_1$, \ldots, $p_n$ are non-associate principal primes and each $\alpha_{i}\ge1$. Each of the $p_{i}^{\alpha_i}$, in addition to being a power of a prime, has other properties,  each of which is subject to generalization. For example,  each $p_i^{\alpha_i}$ is primary and $p_i^{\alpha_i}$ are pairwise comaximal. There exist various generalizations of a (unique) factorization into prime powers in integral domains \cite{BrewerHeinzer2002}. In this paper we consider the comaximal factorization recently introduced by McAdam and Swan \cite{mcdamswan2004}. They defined a nonzero nonunit element $c$ of an integral domain $D$ to be pseudo-irreducible (pseudo-prime) if $d=ab$ ($abR\subset dR$) for comaximal $a$ and $b$  implies that $a$ or $b$ is a unit ($aR\subset dR$ or $bR\subset dR$). A factorization $d=d_i\ldots d_n$ is a (complete) comaximal factorization if each $d_i$ is a nonzero nonunit (pseudo-irreducible) and $d_i$'s are pairwise comaximal. The integral domain $D$ is a comaximal factorization domain if each nonzero nonunit has a complete comaximal factorization.

We call an integral domain atomic if every nonzero nonunit element of $D$ has a factorization into atoms; such a factorization is called an atomic factorization. In \cite{juett2014} it is constructed an atomic domain that is not a comaximal factorization domain. Antimatter domain is an integral domain with no atoms. Among many other interesting results is the one that integral domain is a subring of an antimatter domain that is not a field. Moreover we have the following result:

\begin{theorem}\cite{juett2014}
    Every integral domain is a subring of a one-dimensional antimatter Bezout domain which is a complete comaximal factorization domain.
\end{theorem}

Therefore, the subject of research of this paper is a Bezout domain which is a comaximal factorization domain.

Let us start with the following Henriksen example \cite{henriksen1955}
$$
R=\{z_0+a_1x+a_2x^2+\dots+a_nx^n+\dots \mid z_0\in \mathbb{Z},\; a_i\in \mathbb{Q},\; i=1,2,\dots \}.
$$

This domain is a two-dimensional Bezout domain having a unique prime ideal $J(R)$ (Jacobson radical) of height one and having infinitely many maximal ideals corresponding to the maximal ideals of $\mathbb{Z}$. The elements of $J(R)$ are contained in infinitely many maximal ideals of $R$  while the elements of $R\setminus J(R)$ are contained in only finitely many prime ideals.  This ring is a ring of stable range 2 and $R$ is not a ring of stable range~1.

\begin{definition}
  A ring $R$ is said to have a stable range 1 if for every elements $a,b\in R$ such that $aR+bR=R$ we have $(a+bt)R=R$ for some element $t\in R$. A ring $R$ is said to have a stable range 2 if for every elements $a,b,c\in R$ such that $aR+bR+cR=R$ we have $(a+cx)R+(b+cy)R=R$ for some elements $x,y\in R$.
\end{definition}

The Henriksen example is an example of a ring in which every nonzero nonunit element has only finitely many prime ideals minimal over it. By \cite{mcdamswan2004} in Henriksen's example every nonzero nonunit element has a complete comaximal factorization. The element of $J(R)$ has a complete comaximal factorization $cx^n$, where $c\in\mathbb{Q}$ ($cx^n$ is a pseudo-irreducible element). Moreover, every nonzero element of $J(R)$ is pseudo-irreducible. The elements of $R\setminus J(R)$ are contained in only finitely many  prime ideals and have a complete comaximal factorization corresponding their factorization in $\mathbb{Z}$.

It is shown  \cite{mcdamswan2004} that an integral domain  $R$ has a complete comaximal factorization if either 1) each nonzero nonunit of $R$ has only finitely many minimal primes or 2) each nonzero nonunit of $R$ is contained in only finitely many maximal ideals.

For a commutative ring $R$, let $spec\,R$ and $mspec\, R$ denote the collection of all prime ideals and all maximal ideals of $R$ respectively. The Zariski topology on $spec\, R$ is the topology obtained by taking the collection of sets of the form $D(I)=\{ P\in spec\, R\mid J\nsubseteq P\}$  (respectively $V(I)=\{ P\in spec\, R \mid J\subseteq P\}$), for every ideal $I$ of $R$ as the open (respectively closed) sets.

A topological $space\, X$ is called Noetherian  if every nonempty set of closed subsets of $X$ ordered by inclusion has a minimal element. An ideal $J$ of $R$ is called a $J$-radical ideal if it is the intersection of all maximal ideals containing it. Clearly, $J$-radical ideals of $R$ correspond to closed subset of $mspec\, R$.

When considered as a subspace of $space\, R$, $mspace\, R$ is called a max-spectrum of $R$. So its open and closed subsets are
$$
D(I)=D(I)\cap mspec\,R=\{M\in mspec\, R\mid J\nsubseteq M\}
$$
and
$$
V(I)=V(I)\cap mspec\,R=\{M\in mspec\, R\mid J\subseteq M\},
$$
respectively.
Clearly max-spectrum of $R$ is Noetherian if and only if $R$ satisfies the ascending chain condition for $J$-radical ideals, i.e. $R$ is  a $J$-Noetherian ring~\cite{olberding1998}.


For a commutative Bezout domain $R$ this condition is equivalent to a condition that every nonzero nonunit element of $R$ has only finitely many prime ideals minimal over it \cite{olberding1998}.

Let $R$ be a domain and $a\in R$. Denote by $minspec\,a$ the set prime ideals minimal over $a$.

\begin{lemma}
  Let $R$ be a Bezout domain and $a$ be a  nonzero nonunit element of $R$. Then $a$ is pseudo-irreducible if and only if $\overline{R}=R/aR$ is connected (i.e. its only idempotents are zero and 1).
\end{lemma}

\begin{proof}
  Let $a$ be a pseudo-irreducible element and let $\bar{e}=e+aR\ne\bar{0}$ and $\bar{e}^2=\bar{e}$. Than $e(1-e)=at$ for some element $t\in R$. Let $eR+aR=dR$, where $d\notin U(R)$ and $e=de_0$, $a=da_0$ and $a_0R+e_0R=R$. Then with $e(1-e)=at$ we have $e_0(1-e)=a_0t$. Since $e_0R+a_0R=R$ we have $a_0R+eR=R$, i.e. $a_0R+dR=R$, so $a$ is not a pseudo-irreducible element. We obtained the contradiction, therefore, $\overline{R}$ is connected.

  Let $\overline{R}$ be a connected ring. Suppose that $a$ is not a pseudo-irreducible element, i.e. $a=bc$, where $b\notin U(R)$, $c\notin U(R)$, and $bR+cR=R$. Than $bu+cv=1$ for some elements $u,v\in R$. Then  $\bar{b}^2\bar{u}=\bar{b}$, where $\bar{b}=b+aR$, $\bar{u}=u+aR$. Since $\bar{b}\bar{u}$ is an idempotent and $\overline{R}$ is connected, we have $\bar{b}\bar{u}=\bar{0}$ or $\bar{b}\bar{u}=\bar{1}$. If $\bar{b}\bar{u}=\bar{0}$ then $bu=at$ for some element $t\in R$. Then $bu=bct$, i.e. $u=ct$. It is impossible since $bu+cv=1$ and $c\notin U(R)$. If $\bar{b}\bar{u}=\bar{1}$ then $aR+bR=R$, i.e. $b\in U(R)$. Thus, we proved that $a$ is pseudo-irreducible element.
\end{proof}

We will notice that in a local ring, every nonzero nonunit element is pseudo-irreducible. If $a$ is a element of domain $R$ where $|minspec \,  a|=1$ then $a$ is also pseudo-irreducible. In Henriksen's example $R=\{z_0+a_1x+a_2x^2+\dots+a_nx^n+\dots \mid z_0\in \mathbb{Z},\; a_i\in \mathbb{Q},\; i=1,2,\dots \}$,
$x$ is  pseudo-irreducible. Note that 6 is not  pseudo-irreducible but $x=6\cdot\frac16 x$, i.e. a  nonunit divisor of a pseudo-irreducible element is not pseudo-irreducible. In the case of a neat element it is not true.

\begin{definition}
A commutative ring $R$ is called an elementary divisor ring \cite{kaplansky1949} if for an arbitrary matrix $A$ of order $n\times m$ over $R$ there exist invertible matrices $P\in GL_n(R)$ and $Q\in GL_m(R)$ such that
\begin{description}
  \item[(1)] $PAQ=D$ is diagonal matrix, $D=(d_{ii})$;
  \item[(2)] $d_{i+1,i+1}R\subset d_{ii}R$.
\end{description}
\end{definition}

By \cite{BVZabav2009}, we have the following result.

\begin{theorem}
  Every $J$-Noetherian  Bezout ring is an elementary divisor ring.
\end{theorem}

By \cite{zabboh2017}, for a Bezout domain we have the following result.

\begin{theorem}\label{theor4}
  Let $R$ be a Bezout domain. The following two condition are equivalent:
  \begin{description}
  \item[(1)] $R$ is an elementary divisor ring;
  \item[(2)] for any elements $x,y,z,t\in R$ such $xR+yR=R$ and $zR+tR=R$ there exists an element $\lambda\in R$ such that $x+\lambda y=r\cdot s$, where $rR+zR=R$, $sR+tR=R$ and $rR+sR=R$.
\end{description}
\end{theorem}

\begin{definition}
  Let $R$ be a Bezout domain. An element $a\in R$ is called a neat element if for every elements $b,c\in R$ such that $bR+cR=R$ there exist $r,s\in R$ such that $a=rs$ where $rR+bR=R$, $sR+cR=R$ and $rR+sR=R$. A Bezout domain is said to be of neat range 1 if for any $c,b\in R$ such that $cR+bR=R$ there exists $t\in R$ such that $a+bt$ is a neat element.
\end{definition}

According to Theorem~\ref{theor4} we will obtain the following  result.

\begin{theorem}\label{theor-55}
  A commutative Bezout domain $R$ is an elementary divisor domain if and only if $R$ is a ring of neat range 1.
\end{theorem}

\begin{theorem}
  A nonunit divisor of a neat element of a commutative Bezout domain is a neat element.
\end{theorem}

\begin{proof}
  Let $R$ be a commutative Bezout domain and $a$ be a neat element of $R$. Let $a=xy$. Then for every element $b,c\in R$ such that $bR+cR=R$ there exist $r,s\in R$ such that $a=rs$ where $rR+bR=R$, $sR+cR=R$ and $rR+sR=R$. Let $rR+xR=\alpha R$, $r=\alpha r_0$, $x=\alpha x_0$ for some elements $r_0, x_0\in R$ such that $r_0R+x_0R=R$. Since $a=rs=xy$ we have $\alpha r+0s=\alpha x_0 y$. Then $r_0 s=x_0 y$. Since $r_0R+x_0R=R$, then $r_0u+x_0v=1$ for some elements $u,v \in R$. By $r_0s = x_0y$ we have $r_0su+x_0sv=s$, $x_0(yu+sv)=s$. We have $x=\alpha x_0$ where $\alpha R+bR=R$ and $x_0R+cR=R$ and $\alpha R+x_0R=R$ since $rR\subset \alpha R$ and $sR\subset x_0R$, i.e. $x$  is a neat element.
\end{proof}

\begin{theorem}
  Let $R$ be a $J$-Noetherian  Bezout domain which is not a ring of stable range 1. Then in $R$ there exists an element $a\in R$ such that $R/aR$ is a local ring.
\end{theorem}

\begin{proof}
  By \cite{BVZabav2009}, $R$ is an elementary divisor ring. By Theorem~\ref{theor-55} $R$ is a ring of neat range 1. Since $R$ is not a ring of stable range 1, then in $R$ there exist a  nonunit neat element $a$. Since $R$ is a complete factorization ring \cite{hedayatrostami2018}, then $a$ has a factorization $a=a_1\dots a_n$ where $a$ is a nonzero nonunit pseudo-irreducible element and $a_i$'s are pairwise comaximal.  Since a nonunit divisor of a neat element is  a neat element, we have that $a_i$ is a neat element for all  $i=1,\dots,n$. Since all $a_i$ are pseudo-irreducible, then  $R/aR$ is connected ring. Since $a_i$ is a neat element we have that $R/a_iR$ is connected clean ring. By \cite{ZabavskyGatalevych2015}, $R/a_iR$ is local ring.
\end{proof}


By \cite{Zab14}, any adequate element of a commutative Bezout ring is a neat
element.

\begin{definition}
    An element $a$ of a domain  $R$ is said  to be adequate, if for every element $b\in R$ there exist elements  $r,s\in R$ such that:
\begin{description}
  \item[(1)] $a=rs$;
  \item[(2)] $rR+bR=R$;
  \item[(3)] $s'R+bR\ne R$ for any $s'\in R$ such that $sR\subset s'R\ne R$.
\end{description}
A domain $R$ is called adequate if every nonzero element of $R$ is adequate.
\end{definition}

The most trivial examples of adequate elements are units,  atoms in a ring,  and also square-free elements \cite{BVZmono2012}.

Henriksen observed that in an adequate domain every nonzero prime ideal is contained in an unique maximal ideal \cite{henriksen1955}.

\begin{theorem}
  Let $R$ be a commutative Bezout element and $a$ is non-zero nonunit element of $R$. If $R/aR$ is local ring then $a$ is an adequate element.
\end{theorem}

\begin{proof}
  Let $b\in R$. Since $R/aR$ is local ring, then there exists a unique maximal ideal $M$ such that $a\in M$.

  If $b\in M$ we have $a=1\cdot a$ and for each nonunit divisor $s$ of $a$ we have $sR+bR\in M$, i.e. $sR+bR\ne R$.

  If $b\notin M$, we have $aR+bR=R$.
\end{proof}

\begin{theorem}
  Let $R$ be a $J$-Noetherian Bezout domain which is not a ring of stable range 1. Then in $R$ there exists a nonunit adequate element.
\end{theorem}

\begin{theorem}\label{theor-1-2}
    Let $R$ be a  Bezout domain in which every nonzero nonunit element has only finitely many prime ideals minimal over it. Then  the factor ring $R/aR$ is the direct sum of  valuation rings.
\end{theorem}

\begin{proof}
    Let $P_1, P_2, \ldots, P_n \in\mathrm{minspec}\, aR$.     We consider the factor ring $\overline {R} = R/aR$. We denote $\overline{P}_i = P_i/aR$, where $P_i\in\mathrm {minspec}\, aR$, $i = 1,2, \ldots, n$. Note that $\overline{P}_i\in\mathrm{minspec}\,\overline{R}$ are all minimal prime ideals of the ring $\overline{R}$. By \cite{llsh1974}, the ideals $\overline{P}_i$ are  comaximal in  $\overline {R}$. Obviously, $\mathrm{rad}\,\overline{R}=\bigcap\limits_{i = 1}^n\overline{P}_i$, and by the Chinese remainder theorem we have
    $$
    \overline{R}/\mathrm{rad}\,\overline{R}\cong\overline{R}/\overline{P}_1\oplus\overline{R}/\overline{P}_2\oplus\ldots\oplus\overline{R}/\overline{P}_n. $$
    Since  any prime ideal of $\mathrm{spec}\,aR$ is contained in a unique maximal ideal, then $\overline{R}/\overline{P}_i$ are  valuation rings. Moreover, there exist pairwise orthogonal idem\-po\-tents $\overline{\overline{e}}_1$, \ldots, $\overline{\overline{e}}_n$, where $\overline{\overline{e}}_i\in\overline{R}/\overline{P}_i$ such that $\overline{\overline{e}}_1 + \ldots + \overline{\overline{e}}_n = \overline{1}$. Then, by lifting  the idempotent $\overline{\overline{e}}_i$  modulo $\mathrm{rad}\,\overline{R}$ to pairwise orthogonal idempotents $\overline{e}_1, \ldots, \overline{e}_n \in\overline{R}$ we find that $1-(e_1\ldots +e_n)$ is an idempotent and $1-(e_1+\ldots+e_n)\in\mathrm {rad}\, \overline{R}$, which is possible only if it is  zero. Therefore,
    $$
    \overline{R}=\overline{e}_1\overline{R}\oplus \overline{e}_2\overline{R}\oplus\dots\oplus\overline{e}_n\overline{R}
    $$
    and each $\overline{e}_i\overline{R}$ is a homomorphic image of $\overline{R}$, i.e. a commutative Bezout ring.
    Since   any prime ideal of $\overline{R}$ is contained in a unique maximal ideal, then $\overline{e}_i\overline{R}$ is a valuation  ring.
\end{proof}

    A minor modification of the proof of Theorem~\ref{theor-1-2} gives us the following  result.

\begin{theorem}\label{theor-2-2}
    Let $R$ be a commutative Bezout domain in which any nonzero prime ideal is contained in a finite set of maximal ideals. Then for  any nonzero element $a\in R$ such that the set $\mathrm {minspec}\,aR$ is finite, the factor  ring $\overline{R}=R/aR$ is a direct sum of semilocal rings.
\end{theorem}
\begin{proof}
    According to the notations from  Theorem~\ref{theor-1-2} and its proof, we have
    $$
    \overline {R}=\overline{e}_1\overline{R}\oplus\overline{e}_2\overline{R}\oplus\ldots\oplus\overline{e}_n\overline{R}.
    $$
    Since  any prime ideal of the ring $\overline{R}$ is contained in a finite set of maximal ideals, then $\overline{e}_i\overline{R}$ is a semilocal ring.
\end{proof}

Obviously, if a commutative ring $R$ is a direct sum of valuation rings $R_i$, then $R$ is a commutative Bezout ring. Let $a = (a_1, \ldots, a_n)$, $b = (b_1, \ldots, b_n)$ be any elements of $R$, where $a_i, b_i \in R_i$, $i = 1,2, \ldots, n$. Since $R_i$ is a valuation ring, then $a_i = r_is_i$, where $r_iR + b_iR = R$ and $s_i'R_i + b_iR_i \ne R_i$ for  any non-invertible divisor $s_i'$ of the element $s_i$. If $r = (r_1, \ldots, r_n)$, $s = (s_1, \ldots, s_n)$ then obviously $a = rs$, $rR + bR = R$. For each $i$ such that $s_i'$ is a non-invertible divisor of $s_i \in R_i$, we have $s_iR_i + b_iR_i \ne R_i$. Hence $s'R + bR \ne R$, i.e. $a$ is an adequate element.

A ring $R$ is said to be everywhere adequate if any element of  $R$ is adequate.

Note that, as shown above, in the case of a commutative ring, which is a direct sum of  valuation rings, any element of the ring (in particular zero) is adequate, i.e. this ring is everywhere adequate.

\begin{definition}
    An nonzero element $a$ of a ring $R$ is called an element of almost stable range 1 if the quotient-ring $R/aR$ is a ring of stable range 1.
\end{definition}

    Any ring of stable range 1 is a ring of almost stable 1 \cite{mcgovern2007}. But not every element of stable range 1 is an element of almost stable range 1. For example, let $e$ be a nonzero idempotent of a commutative ring $R$ and $eR+aR=R$. Then $ex+ay=1$ for some elements $x,y\in R$ and $(1-e)ex+(1-e)ay=1-e$, so $e+a(1-e)y=1$. And we have that $e$ is an element of stable range 1 for any commutative ring. However if you consider the ring $R=\mathbb{Z}\times \mathbb{Z}$ and the element $e=(1,0)\in R$ then, as shown above, $e$ is an element of stable range 1, by $R/eR\cong \mathbb{Z}$,  and $e$ is not an element of almost stable range 1. Moreover, if $R$ is a commutative principal ideal domain (e.g. ring of integers), which is not of stable range 1, then every nonzero element of $R$ is an element of almost stable range 1.

\begin{definition}
    A commutative ring in which every nonzero element is an element of almost stable range 1 is called a ring of almost stable range 1.
\end{definition}

      The first  example of a ring of  almost stable range 1  is a ring of stable range 1. Also, every commutative principal ideal ring which is not a ring of stable range 1 (for example, the ring of integers) is a ring of almost stable range 1 which is not a ring of stable range 1.

    We note that a semilocal ring is an example of a ring of stable range 1. Moreover, a direct sum of rings of stable range 1 is a ring of stable range 1. As a consequence, we obtain the result from the previous theorems.

\begin{theorem}\label{theor-3-2}
    Let $R$ be a  Bezout domain in which every nonzero nonunit element has only finitely many prime ideals minimal over it.  Then the factor ring $R/aR$ is everywhere adequate if and only if $R$ is a direct sum of  valuation rings.
\end{theorem}

\begin{theorem}\label{theor-4-2}
    Let $R$ be a  Bezout domain in which every nonzero nonunit element has only finitely many prime ideals minimal over it  and any nonzero prime ideal $\mathrm{spec}\,aR$ is contained in a finite set of maximal ideals. Then $a$ is an element of almost stable range 1.
\end{theorem}

\textbf{Open Question}. Is it true that every  commutative Bezout domain in which any non-zero prime ideal is  contained in a finite set of maximal ideals is an elementary divisor ring?

\end{document}